\documentclass{ws-ijnt}
\begin{document}

\author{SAMANTHA C. MOORE}
\address{Department of Mathematics, University of North Carolina at Chapel Hill, 3250 Phillips Hall \#3\\
Chapel Hill, 27599, United States of America\\
scasya@live.unc.edu}
\title{ON THE ZEROES OF HALF-INTEGRAL WEIGHT EISENSTEIN SERIES ON $\Gamma_0(4)$}

\maketitle
\begin{history}
\received{(14 January 2017)}
\accepted{(18 July 2017)}
\end{history}

\markboth{Samantha C. Moore}
{On the Zeroes of Half-Integral Weight Eisenstein Series on $\Gamma_0(4)$}

\begin{abstract}
We locate all but $O(\sqrt{k\log{k}})$ zeroes of the half integral weight Eisenstein series $E_\infty(z,k)$ of $\Gamma_0(4)$ for $k$ sufficiently large. To do this, we relate $E_\infty(z,k)$ to $\Gamma_0(4)$'s other Eisenstein series, $E_0(z,k)$ and $E_\frac{1}{2}(z,k)$, which are easier to study in the regions of which zeroes occur. We will use trigonometric approximations of $E_0(z,k)$ and $E_\frac{1}{2}(z,k)$ in order to locate the zeroes.
\end{abstract}



\section{Overview}\label{sec:Intro}
Let $k\geq 5$ be an odd integer. In this paper we study the zeroes of the weight  $\frac{k}{2}$ Eisenstein series, $E_\infty(z)$, of $\Gamma_0(4)$, which is defined as 
\begin{equation} E_\infty(z)=e^{\frac{\pi ik}{4}}\sum\limits_{\substack{(d,2c)=1\\ c>0}}\frac{G\big(\frac{-d}{4c}\big)^k}{(4cz+d)^{\frac{k}{2}}},\label{eq:EinfDef}\end{equation}
where $G\big(\frac{m}{n}\big)=\frac{g(\frac{m}{n})}{|g(\frac{m}{n})|}$ and $g(\frac{m}{n})=\sum\limits_{a\mod n} e^\frac{2\pi ia^2m}{n}$ is the Gauss sum [2]. Note that, though $E_\infty(z)$ depends on $k$ as well as $z$, we omit this dependence from the notation.

This Eisenstein series is interesting to study due to its weight; many authors have studied the zeroes of integral weight modular forms, including [5] and [7]. On the other hand, there have been studies of half integral weight modular forms (such as that of [1]), but none of these have been on the zeroes of their Eisenstein series. As the simplest half integral weight modular form, $E_\infty(z)$ is thus a natural choice to study.

Before stating our results, we first state some basic properties of $\Gamma_0(4)$ and its Eisenstein series. Many of these properties can be found in [2] and [4]. 

In order to study the zeroes of $E_\infty$ on all of $\mathbb{H}$, it suffices to find the zeroes of $E_\infty(z)$ in a fundamental domain. Our choice of fundamental domain, $F_\infty$,  of $\Gamma_0(4)$ is the union of six regions [3], as shown in Fig. \ref{fig:Finf}.

\begin{figure}[h]\vspace{-2cm}
\centering \includegraphics[scale=.35]{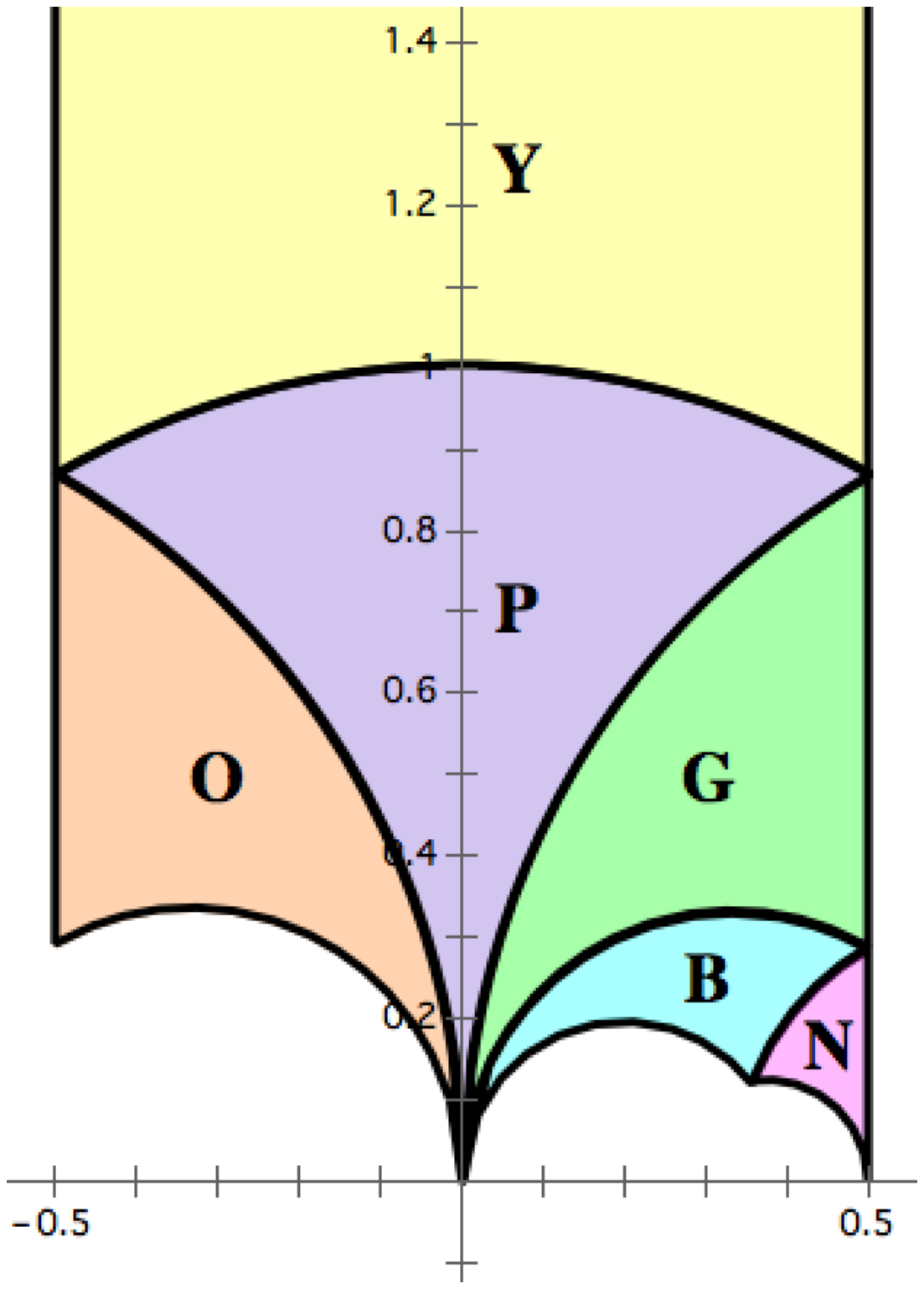}\vspace{-2cm}
\caption{$F_{\infty}$\label{fig:Finf}}
\end{figure}

There are three cusps of $F_\infty$: $0, \frac{1}{2},$ and $\infty$. We have already seen the Eisenstein series associated to the cusp $\infty$. The other two Eisenstein series (see [2]) are defined as

\begin{equation}E_0(z)=\sum_{\substack{(u,2v)=1\\ u>0}}^{ } \frac{(\frac{-v}{u}) \epsilon_u^k}{(uz+v)^{\frac{k}{2}}}, \hspace{1cm} E_\frac{1}{2}(z)=e^\frac{-\pi ik}{4}\sum_{\substack{(2c,d)=1\\ d>0}}\frac{G(\frac{d-2c}{8d})^k}{(dz+c)^{\frac{k}{2}}}.
\label{eq:EzeroAndEonehalfDef}\end{equation}

Here $(\frac{a}{b})$ is the Jacobi symbol and 
\[
  \epsilon_u= \left\{\def\arraystretch{1.2}%
  \begin{array}{@{}c@{\quad}l@{}}
   1, & u\equiv1\mod4\\
	i, & u\equiv3\mod4.\\
  \end{array}\right. 
\]
We will also find it useful to use the Fourier expansion of $E_0(z)$ (see [4]), which is 
\begin{equation} E_0(z)=2^{\frac{k}{2}}\sum\limits_{l=1}^{\infty}b_\ell q^\ell\label{eq:EzeroFourier}\end{equation}
where $q=e^{2\pi iz}$, and 
\begin{equation} b_\ell=\frac{\pi^{\frac{k}{2}}\ell^{\frac{k}{2}-1}}{\Gamma(\frac{k}{2})e^{\frac{\pi ik}{4}}}\sum\limits_{n>0\  odd }\epsilon_n^k n^{-\frac{k}{2}}\sum\limits_{j=0}^{n-1} \Big(\frac{j}{n}\Big)e^{-\frac{2\pi i\ell j}{n}}.\label{eq:blDef}\end{equation}

The following relations exist between  $E_\infty(z)$ and  $E_0(z)$ and  $E_\frac{1}{2}(z)$ [2]:
\begin{equation}E_0\Big(-\frac{1}{4z}\Big)=(4z)^{\frac{k}{2}}i^{-k}E_\infty(z),\label{eq:EzeroEinfRelation}\end{equation}
\begin{equation}E_{\frac{1}{2}}\Big(z+\frac{1}{2}\Big)=2^{\frac{k}{2}}(2z+1)^{-\frac{k}{2}}E_\infty\Big(\frac{z}{2z+1}\Big).\label{eq:EonehalfEinfRelation}\end{equation}

These allow us to study zeroes of $E_\infty(z)$ near the cusps $0$ and $\frac{1}{2}$ by studying zeroes of $E_0(z)$ and  $E_\frac{1}{2}(z)$ in regions with $\text{Im}(z)\geq\frac{1}{2}$. We study  $E_0(z)$ on $F_0$ and $E_\frac{1}{2}(z)$ on $F_\frac{1}{2}$, where $F_0$ and  $F_\frac{1}{2}$ are the relevant fundamental domains that correspond to $F_\infty$. These are shown in Fig. \ref{fig:FzeroAndFonehalf}.

\begin{figure}[ph]\vspace{-2.1cm}
\begin{center}
\includegraphics[scale=.35]{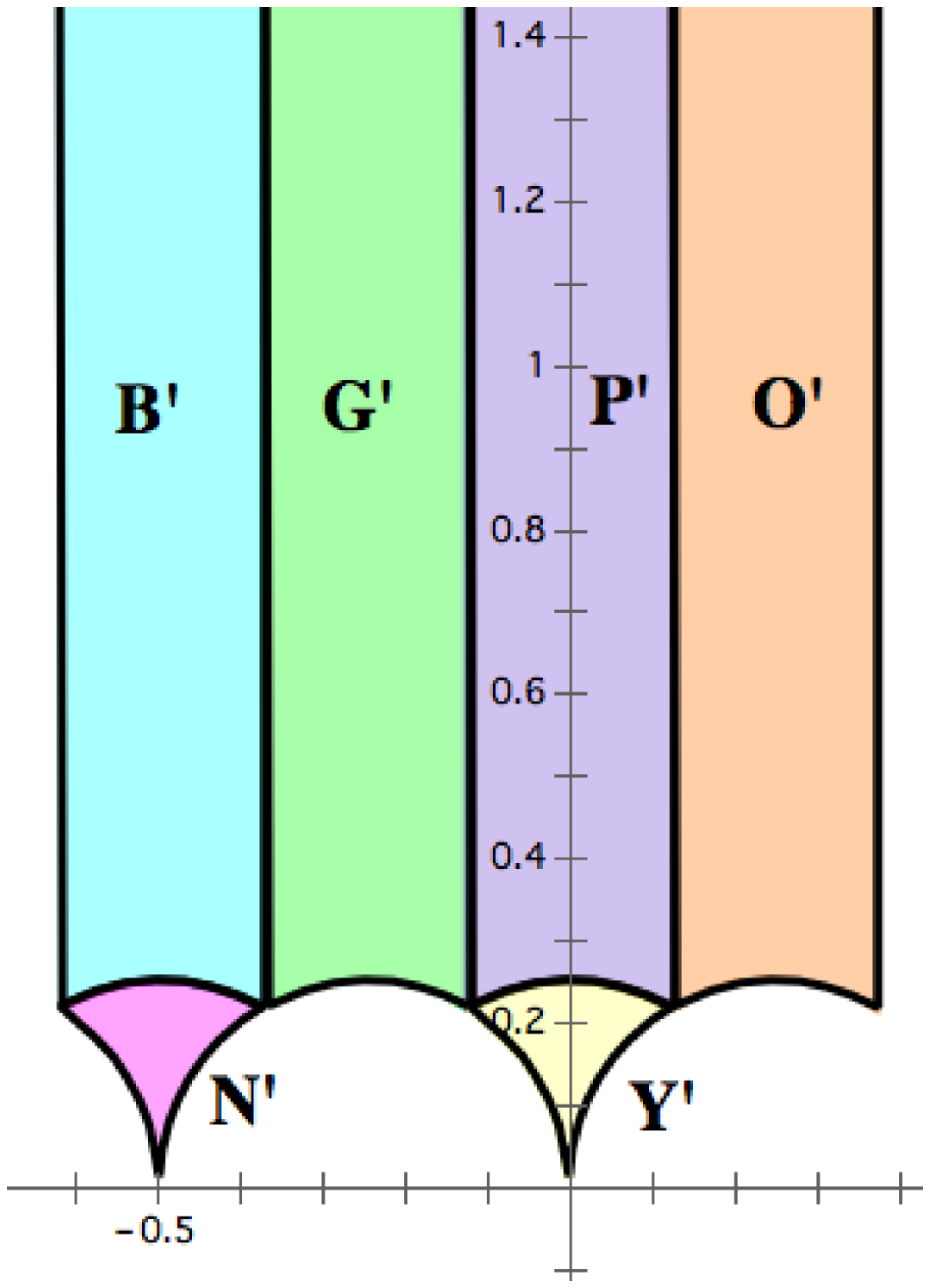}\hspace*{-2cm}\includegraphics[scale=.345]{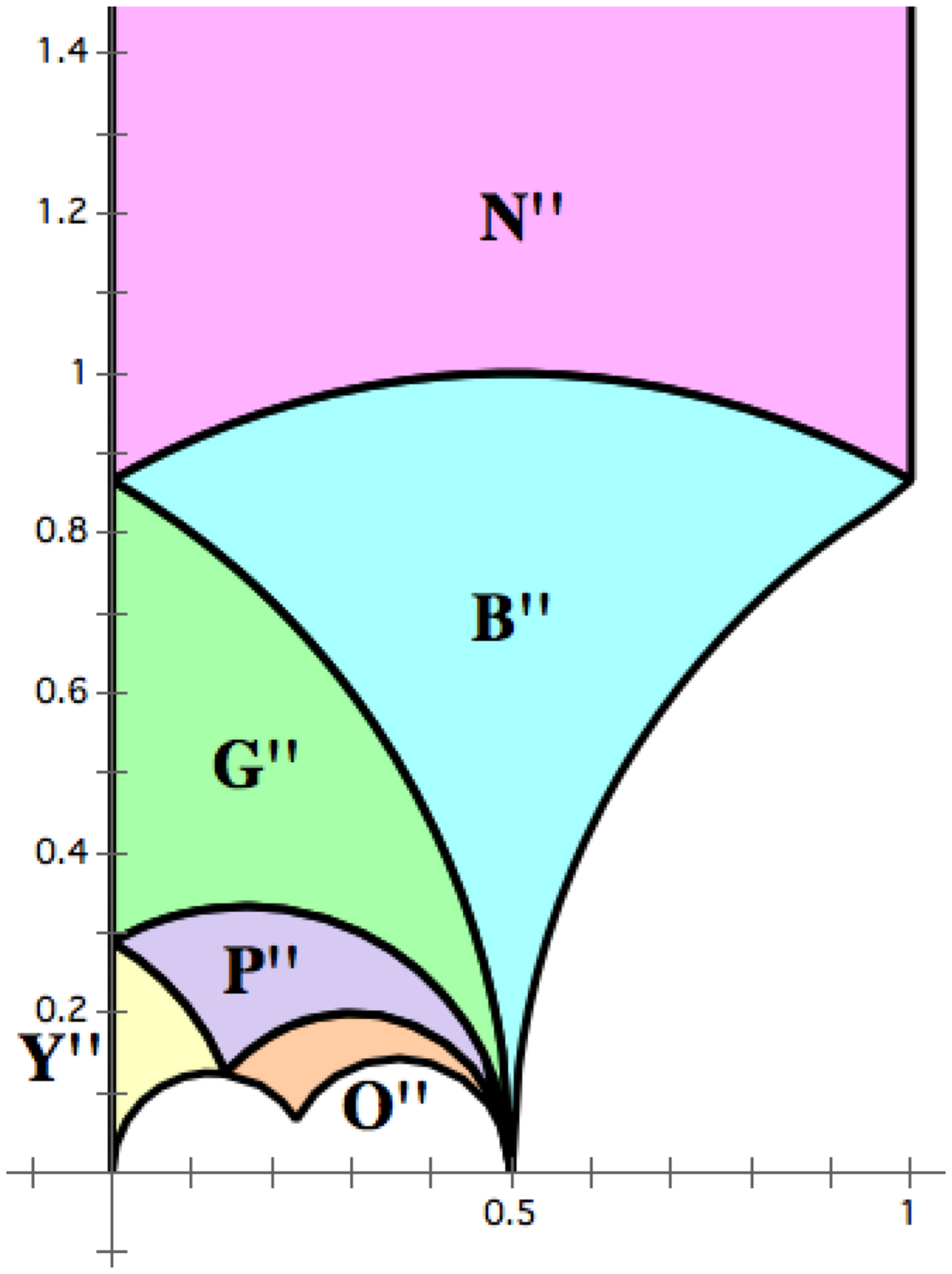}\end{center}\vspace{-2cm}
\caption{$F_0$ (left) and $F_\frac{1}{2}$ (right)}\label{fig:FzeroAndFonehalf}
\end{figure}

 Based on the studies by [5] and [7], one might guess that the zeroes of $E_\infty(z)$ would lie along the boundaries of the fundamental domains. In our case, we have the following conjecture:

 \begin{conjecture}
 All but one of the zeroes of $E_\infty(z)$ in $F_\infty$ lie on the circle $|z-\tfrac{1}{4}|= \tfrac{1}{4}$, which is within the regions $N$ and $B$. The final zero lies in $P$ along the line $x=0$.
 \label{conj}
 \end{conjecture}
 
It is worth noting that, though these zeroes are not on the boundaries of $F_\infty$,  it is possible to choose a fundamental domain in which this is the case (see [1]). Also, a more technical note: in Conjecture \ref{conj} and throughout the paper, we use the term zeroes to refer specifically to the non--trivial zeroes; here, the trivial zeroes are those at the cusps $0$ and $\frac{1}{2}$.

This paper will not address the zero in $P$. To study the zeroes in the regions $N$ and $B$, we study the corresponding zeroes of $E_0(z)$ and $E_\frac{1}{2}(z)$. More specifically, the zeroes in these regions lie on the line $x=-\frac{1}{2}$ in $N'$ and $B'$ and the line $x=\frac{1}{2}$ in $N''$ and $B''$, which correspond to the circle $|z-\tfrac{1}{4}|= \tfrac{1}{4}$ in $F_\infty$. For example, Fig. \ref{fig:K15zeroes} shows the zeroes in each fundamental domain when $k=15$. The only zero not seen in any of the pictures is the zero in $P$; this is because in each of the pictures it is too low or to high to be seen. For instance, the zero is located at $\approx5.881i$ in $F_0$.
\begin{figure}[!h]\vspace{-2cm}
\begin{center} \hspace*{-1.3cm}\includegraphics[scale=.32]{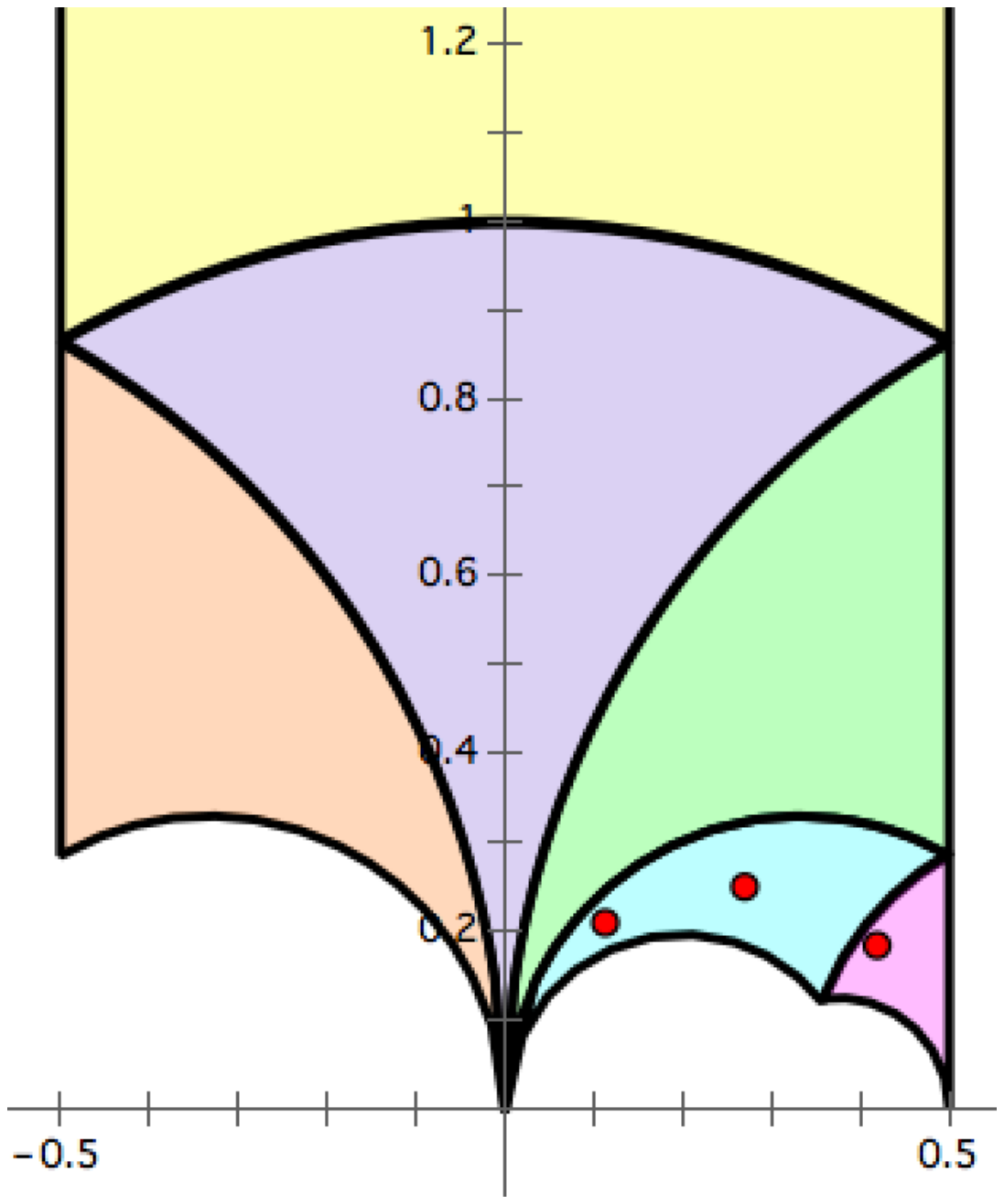}\hspace*{-2.8cm}\includegraphics[scale=.32]{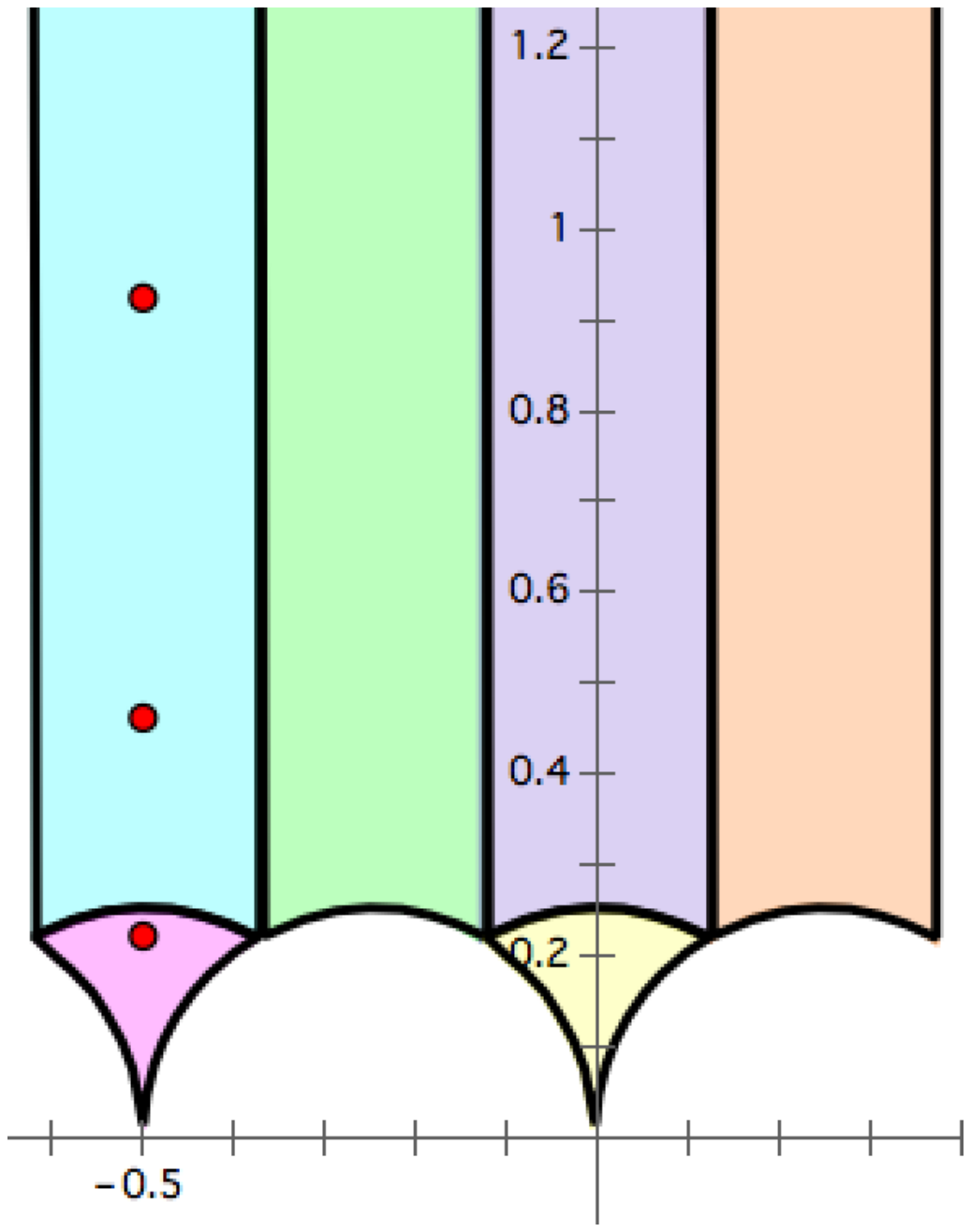}\hspace*{-2.8cm} \includegraphics[scale=.32]{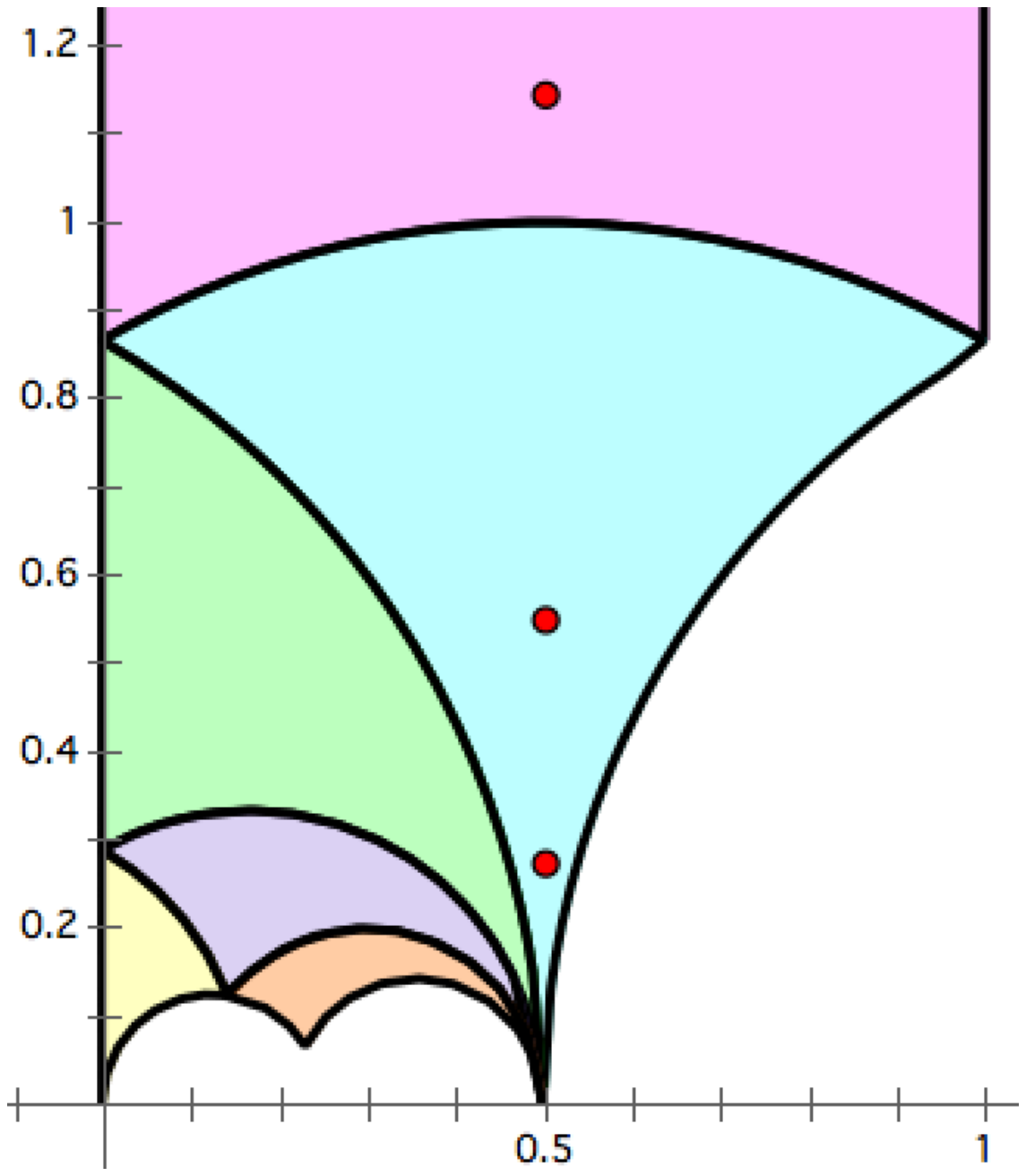}\end{center}\vspace{-2cm}
\caption{Location of zeroes for $k=15$}
\label{fig:K15zeroes}
\end{figure}

As such, in the following sections we consider $E_ 0(-\tfrac{1}{2}+iy)$ and $E_ \frac{1}{2}(\tfrac{1}{2}+iy)$. In order to restrict our studies to regions with not--too--small imaginary part, we utilize the following additional automorphy relations:
\begin{equation}E_{\frac{1}{2}}(z)=i^{-k}E_{\frac{1}{2}}(z+1),\label{eq:EonehalfEonehalfRelation}\end{equation}
\begin{equation}E_{\frac{1}{2}}(z)=(2z+1)^{-\frac{k}{2}}E_0\Big(\frac{z}{2z+1}\Big).\label{eq:EzeroEonehalfRelation}\end{equation}
 
By relation \eqref{eq:EzeroEonehalfRelation} and the periodicity of $E_0(z)$, $E_\frac{1}{2}(\tfrac{1}{2}+iy)$ where $y\geq\frac{1}{2}$ corresponds to $E_0(-\frac{1}{2}+iy')$ where $y'\leq\frac{1}{2}$. As such, it suffices to study only $y\geq\frac{1}{2}$ in each case so as to count each zero only once. 

\begin{theorem}
For $k$ sufficiently large, $E_0(-\tfrac{1}{2}+iy)$  and $E_\frac{1}{2}(\tfrac{1}{2}+iy)$ each have at least $\frac{k}{8}-O(\sqrt{k\log{k}})$ zeroes for $y\geq\frac{1}{2}$. 
\label{EzeroThm}
\end{theorem} 
Proving this theorem will be the focus of the rest of the paper. The following corollary relates this theorem to zeroes of $E_\infty(z)$ and makes considerable progress toward proving Conjecture \ref{conj}.
\begin{corollary}
For $k$ sufficiently large, all but at most $O(\sqrt{k\log{k}})$ zeroes of $E_\infty(z)$ lie on the circle $|z-\tfrac{1}{4}|= \tfrac{1}{4}$ in $F_\infty$.
\label{corol}
\end{corollary} 
Note that this corollary shows asymptotically that we have found all of the zeroes on the circle.
\begin{proof}
The valence formula [4] for $\Gamma_0(4)$ tells us that, at most, $E_\infty(z)$ has $\lfloor{\frac{k}{4}}\rfloor$ zeroes (here, some care must be taken to count zeroes at the cusps with appropriate weights). Thus, the corollary will follow directly from Theorem \ref{EzeroThm}.
\end{proof}

To prove Theorem \ref{EzeroThm}, we need a few simpler results. First, we prove that $e^\frac{\pi ik}{4}E_0(-\tfrac{1}{2}+iy)$ is a real--valued function. We then find a real--valued trigonometric approximation of  $e^\frac{\pi ik}{4}E_0(-\tfrac{1}{2}+iy)$, which we denote as $M_0$. We bound the error of this approximation for large $k$ and $y\leq\frac{\sqrt{k}}{\sqrt{2\log{k}}}$, . Similarly, we then prove that $E_\frac{1}{2}(\tfrac{1}{2}+iy)$ is a real--valued function with real--valued trigonometric approximation $M_\frac{1}{2}$. We again bound this error under the same restrictions for $y$ and $k$. 

In the final section, we use the intermediate value theorem to determine zeroes of $M_0$ and $M_\frac{1}{2}$. Finally, by our bounds on the error of $M_0$ and $M_\frac{1}{2}$ in comparison to $E_0(-\tfrac{1}{2}+iy)$ and  $E_\frac{1}{2}(\tfrac{1}{2}+iy)$, we prove that each of the zeroes found in the last step correspond to a zero of $E_0(-\tfrac{1}{2}+iy)$ and  $E_\frac{1}{2}(\tfrac{1}{2}+iy)$.

%%%%%%%%%
%%%%%%%%%
%%%%%%%%%
\section{Previous Results}\label{sec:Background}
This paper follows methods from [5], which we now summarize. Rankin and Swinnerton-Dyer showed that all the zeroes of the even weight $k$ Eisentstein Series $E_k(z)$ of $SL_2(\mathbb{Z})$ lie on the bottom arc of the standard fundamental domain for $SL_2(\mathbb{Z})\backslash \mathbb{H}$. Let $k=12n+s$ where $s\in\{0,4,6,8,10,14\}$ and

\begin{equation}F_k(\theta)=e^{\frac{ik\theta}{2}}E_k(e^{i\theta})=\frac{1}{2}\sum\limits_{(c,d)=1}(ce^{\frac{i\theta}{2}}+de^{-\frac{i\theta}{2}})^{-k}.\label{eq:RSD}\end{equation}

Note that $F_k(\theta)$ converges and is real--valued for $0<\theta<\pi$. Pulling out the largest terms (those with $c^2+d^2=1$), they simplify $F_k(\theta)$ to $F_k(\theta)=2\cos{(\frac{k\theta}{2})}+R_1$, where $R_1$ is the sum of the remaining terms. Rankin and Swinnerton-Dyer show that $|R_1|<2$. Let $\theta_m=\frac{2m\pi}{k}$ where $m$ is an integer such that $\frac{k}{4}\leq m\leq\frac{k}{3}$. As the sign of $F_k(\theta_m)$ changes with each increase of $m$, the intermediate value theorem can be applied to show a zero exists between each value of $m$. Using the restrictions on $m$ and $\theta_m$, $n$ zeroes are found. By the valence formula for $E_k$, it is known that there are at most $n$ zeroes. Thus, all of the zeroes have been accounted for. 

Reitzes, Vulakh, and Young [7] further expanded on these methods. They studied the zeroes of the weight $k+\ell$ cusp forms $\Delta_{k,\ell}=E_kE_\ell-E_{k+\ell}.$ In this case, all of the zeroes lie on the boundaries of the standard fundamental domain (including the vertical boundaries, unlike in [5]). They begin by using the same method as [5], however they find that their approximations are valid for only $y\ll \sqrt{\frac{k}{\log{k}}}$, which is the same limitation that we encounter in this paper. For simplicity, they choose to use this approximation for $y\leq k^{\frac{2}{5}}$. As some zeroes lie above this $y$ value, they must develop two other techniques. The first uses the Fourier expansion, rather than the $cz+d$ expansion, to approximate $\Delta_{k,\ell}$ for  $y\geq k^{\frac{3}{5}}$. The second technique approximates $\Delta_{k,\ell}$ for $k^{\frac{2}{5}}\leq y\leq k^{\frac{3}{5}}$ by including more terms in their original $cz+d$ approximation. 

Though we do not use the techniques from [7], we note them as a possibility for future research; it seems likely that these methods could be used to study larger $y$ values in order to prove Conjecture \ref{conj} in full, however, the details become quite intricate.

%%%%%%%%%
%%%%%%%%%
%%%%%%%%%
\section{Fundamental Domains}\label{sec:FundDomains}
In this section we determine fundamental domains for $E_\infty$, $E_0$, and $E_\frac{1}{2}$.

The fundamental domain, $F_\infty$,  of $\Gamma_0(4)$ is the union of six regions [3], as shown in Fig. \ref{fig:Finf}. These regions are defined as follows:
	\begin{center}
	$Y=\{z\in\mathbb{H}: |\text{Re}(z)| \leq\frac{1}{2}, |z|\geq 1\},$\\
	$P=\{z\in\mathbb{H}: |z|<1, |z+1|\geq 1, |z-1|\geq 1\},$\\
	$O=\{z\in\mathbb{H}:  \text{Re}(z) \geq -\frac{1}{2}, |z+1|<1, |z+\frac{1}{3}|\geq\frac{1}{3}\},$\\
	$G=\{z\in\mathbb{H}: \text{Re}(z) \leq\frac{1}{2}, |z-\frac{1}{3}|\geq\frac{1}{3}, |z-1|< 1\},$\\
	$B=\{z\in\mathbb{H}:  |z-\frac{1}{3}|<\frac{1}{3}, |z-\frac{1}{5}|\geq\frac{1}{5}, |z-\frac{2}{3}|\geq\frac{1}{3}\},$\\
	$N=\{z\in\mathbb{H}: \text{Re}(z) \leq\frac{1}{2}, |z-\frac{2}{3}|<\frac{1}{3}, |z-\frac{3}{8}|\geq\frac{1}{8}\}.$\\
	\end{center}
Note that $F_\infty=Y\cup P\cup O\cup G\cup B\cup N$ is not quite a true fundamental domain, as it is possible for two distinct points in $F_\infty$ to be $\Gamma_0(4)$-equivalent if and only if both points lie on the outer boundaries. However, for our purposes this definition is sufficient.

Let $S(z)=-\frac{1}{z}$ and $T(z)=z+1$. The boundaries of $F_ 0$, the fundamental domain that corresponds to $F_\infty$ for $E_ 0(z)$, (shown in Fig. \ref{fig:FzeroAndFonehalf}) can be found via $z\mapsto \tfrac{1}{4}S(z)$. Similarly, the boundaries of  $F_ \frac{1}{2}$, the fundamental domain that corresponds to $F_\infty$ for $E_ \frac{1}{2}(z)$, (also shown in Fig. \ref{fig:FzeroAndFonehalf}) can be found via $z\mapsto -ST^{2}S(z)+\frac{1}{2}$.

%%%%%%%%%
%%%%%%%%%
%%%%%%%%%
\section{Approximations of $E_0$} \label{sec:Ezero}
%%%
%%%
%%%
%In this next section, we prove that $e^\frac{\pi ik}{4}E_0(-\tfrac{1}{2}+iy)$ is real--valued. We then approximate this function by a real--valued trigonometric function. Lastly, we bound the error of this approximation. 

\subsection{Properties of $e^\frac{\pi ik}{4}E_0(-\tfrac{1}{2}+iy)$}\label{sub:Rzero}
\begin{proposition}
The function $e^\frac{\pi ik}{4}E_0(-\tfrac{1}{2}+iy)$ is real--valued. 
\label{prop:Rzero}
\end{proposition}
\begin{proof}
To prove this, we use the Fourier expansion given in \eqref{eq:EzeroFourier} instead of the $uz+v$ expansion given in \eqref{eq:EzeroAndEonehalfDef}. Letting $z=-\frac{1}{2}+iy$,  we have $q^\ell=(-e^{-2\pi y})^\ell$, which is real--valued. Thus, we need to prove that $e^\frac{\pi ik}{4} b_\ell$ is real for all $\ell$. We do this in two cases.\\
\textbf{Case 1:} $\ell$ squarefree.\\
When $\ell$ is squarefree, [4] (see p. 189) simplifies $b_\ell$ to 

$$b_\ell=\frac{\pi^{\frac{k}{2}}\ell^{\frac{k}{2}-1}}{\Gamma(\frac{k}{2})e^{\frac{\pi ik}{4}}}\sum\limits_{\substack{n_0>0,\\  \text{odd, and}\\ \text{squarefree}}}\sum\limits_{\substack{n_1\mid \ell \\ n_1 \text{odd}}}\epsilon_{n_0n_1^2}^{k+1}(n_0n_1^2)^{-\frac{k}{2}} \Big(\frac{-\ell}{n_0}\Big)\sqrt{n_0}\mu(n_1)n_1,$$ 
where $ \mu(n_1)$ is the M\"{o}bius function. Note that $\epsilon_{n_0n_1^2}^{k+1}=\pm 1$ as $k$ is odd. Thus, $e^{\frac{\pi ik}{4}}b_\ell\in\mathbb{R}$ for all $\ell$ squarefree.\\
\textbf{Case 2:} $\ell$ not squarefree.\\
Let $\ell=p^{2v}\ell_0$ and $p^2\nmid \ell_0$. By [4] (see pp.195-199), 

\[
  \frac{b_\ell}{b_{\ell_0}} = \left\{\def\arraystretch{1.2}%
  \begin{array}{@{}c@{\quad}l@{}}
  2^{(k-2)v}, & p=2\\
	\sum\limits_{h=0}^{v} p^{h(k-2)}, & p \text{ odd prime }  p\mid \ell_0\\
		\sum\limits_{h=0}^{v} p^{h(k-2)}-\chi_{(-1)^\lambda \ell_0}(p)p^{\lambda-1}\sum\limits_{h=0}^{v} p^{h(k-2)}, & p \text{ odd prime }  p\nmid \ell_0.\\
  \end{array}\right.
\]
where $\lambda=\frac{k-1}{2}$ and $\chi_{(-1)^\lambda \ell_0}=\Big(\frac{-1}{p}\Big)^\lambda\Big(\frac{\ell_0}{p}\Big)$. Thus, $\frac{b_\ell}{b_{\ell_0}}\in\mathbb{R}$ . By iterating this process, we see that $\frac{b_\ell}{b_{\ell_*}}\in\mathbb{R}$ where $\ell_*$ is squarefree. Furthermore, $e^{\frac{\pi ik}{4}}b_\ell=Ce^{\frac{\pi ik}{4}}b_{\ell_*}$ where $C\in\mathbb{R}$.  Thus $e^{\frac{\pi ik}{4}}b_\ell$ is real--valued for all $\ell$. 
\end{proof}

Now we return to the $uz+v$ expansion of $E_0(z)$ given in Eq. \eqref{eq:EzeroAndEonehalfDef}; in the next two subsections, we aim to find a trigonometric approximation for this infinite sum that is accurate for $k$ large enough. Thus, consider the following terms of Eq. \eqref{eq:EzeroAndEonehalfDef}, which are the largest terms when $x=-\frac{1}{2}$. Letting $u=1$ and $v=0$, we get the term $\frac{1}{z^{\frac{k}{2}}}$. Letting $u=1$ and $v=1$, we get the term $\frac{1}{(z+1)^{\frac{k}{2}}}$. Though it is not obvious at this point, we want to multiply our approximation by $e^\frac{\pi ik}{4}r^\frac{k}{2}$ where $r=|-\frac{1}{2}+iy|$. Thus, let

\begin{equation} M_0(z)=e^\frac{\pi ik}{4}r^\frac{k}{2}\Big(\frac{1}{z^{\frac{k}{2}}}+\frac{1}{(z+1)^{\frac{k}{2}}}\Big).\label{eq:M0Def}\end{equation}

The following two sums include all of the terms left to be bounded. Again, we want to multiply the sums by $e^\frac{\pi ik}{4}r^\frac{k}{2}$, giving
\begin{equation}J_1=e^\frac{\pi ik}{4}r^\frac{k}{2}\sum\limits_{v\neq 0,1} \frac{(\frac{-v}{1})\epsilon_1^{k}}{(z+v)^{\frac{k}{2}}}, \hspace{2cm} J_2=e^\frac{\pi ik}{4}r^\frac{k}{2}\sum\limits_{\substack{(u,2v)=1\\ u>1}} \frac{(\frac{-v}{u})\epsilon_u^{k}}{(uz+v)^{\frac{k}{2}}}.\label{eq:J1AndJ2Def}\end{equation}

By these definitions,  $e^\frac{\pi ik}{4}r^\frac{k}{2}E_0(z)=M_0(z)+J_1+J_2$.

\begin{proposition} 
\label{lem:EzeroApproxes}
Let $\delta=\arctan{(2y)}$ and consider $\frac{1}{2}\leq y\leq\frac{\sqrt{k}}{\sqrt{2\log{k}}}$ . For $k$ large,
\[
  M_0(-\tfrac{1}{2}+iy) =  \left\{\def\arraystretch{1.2}%
  \begin{array}{@{}c@{\quad}l@{}}
   2\cos{(\frac{\delta k}{2}}-\frac{\pi}{4}), & k\equiv 1 $ mod $ 4\\
	2\cos{(\frac{\delta k}{2}}+\frac{\pi}{4}), & k\equiv 3 $ mod $ 4,\\
  \end{array}\right. 
\]
$|J_1|=o(1)$, and $|J_2|\ll(\frac{8}{81})^\frac{k}{4}$.
\end{proposition}
\begin{proof}
The proof is given in the following two subsections.

%%%
%%%
%%%
\subsection{Simplifying $M_0(-\tfrac{1}{2}+iy)$}\label{Mzero}
We convert to polar notation by letting $\frac{1}{2}+iy=re^{i\delta}$ where $\delta=\arctan(2y)$. Note that this implies $-\frac{1}{2}+iy=re^{i(\pi-\delta)}$. Thus,

\begin{equation}M_0(-\tfrac{1}{2}+iy)=e^\frac{\pi ik}{4}r^\frac{k}{2}\Big(\frac{1}{(re^{i\delta})^{\frac{k}{2}}}+\frac{1}{(re^{i(\pi-\delta)})^{\frac{k}{2}}}\Big)=e^\frac{\pi ik}{4}\Big(e^{\frac{-i\delta k}{2}}+e^{\frac{-i(\pi-\delta)k}{2}}\Big).\label{eq:eq14}\end{equation}
As $e^{ix}=\cos{(x)}+i\sin{(x)}$, this becomes
\begin{equation}M_0(-\tfrac{1}{2}+iy)=e^\frac{\pi ik}{4}\big[\cos{(\tfrac{-\delta k}{2})}+i\sin{(\tfrac{-\delta k}{2})}+\cos{(\tfrac{-k\pi}{2} +\tfrac{k\delta}{2})}+i\sin{(\tfrac{-k\pi}{2} +\tfrac{k\delta}{2})}\big].\label{eq:eq15}\end{equation}

When $k\equiv 1$ mod $4$, this can be further simplified by the the even/odd properties of cosine and sine as well as angle identities $\cos{\big(x-\frac{\pi}{2}}\big)=\sin{(x)}$ and $\sin{\big(x-\frac{\pi}{2}}\big)=-\cos{(x)}$. Then apply the identity $\cos{(x)}+\sin{(x)}=\sqrt{2}\cos{\big(x-\frac{\pi}{4}}\big)$ to obtain

\begin{equation}M_0(-\tfrac{1}{2}+iy)=2\cos{(\tfrac{\delta k}{2}}-\tfrac{\pi}{4}).\label{eq:eq17}\end{equation}

%\begin{equation}M_0(-\tfrac{1}{2}+iy)=e^\frac{\pi ik}{4}\big[\cos{(\tfrac{\delta k}{2})}-i\sin{(\tfrac{\delta k}{2})}+\sin{(\tfrac{k\delta}{2})}-i\cos{(\tfrac{k\delta}{2}}\big].\label{eq:eq16}\end{equation}
%We can simplify further and use the identity $\cos{(x)}+\sin{(x)}=\sqrt{2}\cos{\big(x-\frac{\pi}{4}}\big)$ to obtain

%\begin{equation}M_0(-\tfrac{1}{2}+iy)=e^\frac{\pi ik}{4}(1-i)[\cos{(\tfrac{\delta k}{2})}+\sin{(\tfrac{\delta k}{2})}]=2\cos{(\tfrac{\delta k}{2}}-\tfrac{\pi}{4}).\label{eq:eq17}\end{equation}

Similarly, when $k\equiv 3\mod 4$, 
\begin{equation}M_0(-\tfrac{1}{2}+iy)=2\cos{(\tfrac{\delta k}{2}}+\tfrac{\pi}{4}).\label{eq:eq19}\end{equation}
Thus, 
\[
  M_0(-\tfrac{1}{2}+iy) =  \left\{\def\arraystretch{1.2}%
  \begin{array}{@{}c@{\quad}l@{}}
   2\cos{(\frac{\delta k}{2}}-\frac{\pi}{4}), & k\equiv 1 $ mod $ 4\\
	2\cos{(\frac{\delta k}{2}}+\frac{\pi}{4}), & k\equiv 3 $ mod $ 4.\\
  \end{array}\right. 
\]

Note that $M_0(-\tfrac{1}{2}+iy)$ is real--valued. This will be used in Section \ref{sec:Sample}.

%%%
%%%
%%%
\subsection{Bound on Remaining Terms of $E_0$}\label{sub:BoundEzero}
Next, we wish to show that $|J_1|=o(1)$ and $|J_2|\ll(\frac{8}{81})^\frac{k}{4}$ for $k$ large and $\frac{1}{2}\leq y\leq\frac{\sqrt{k}}{\sqrt{2\log{k}}}$.

We begin by bounding $J_1$. By the triangle inequality,

\begin{equation}|J_1|\leq\sum\limits_{v\neq 0,1} \frac{|r^{\frac{k}{2}}e^{\frac{\pi i k}{4}}(\frac{-v}{1})\epsilon_1^{k}|}{|z+v|^{\frac{k}{2}}}=\sum\limits_{v\neq 0,1} \frac{r^{\frac{k}{2}}}{|z+v|^{\frac{k}{2}}}.\label{eq:eq23}\end{equation}
When $x=-\frac{1}{2}$, $|z+v|=|z-v+1|$, giving

\begin{equation}|J_1|\leq2\sum\limits_{v\geq2} \frac{r^{\frac{k}{2}}}{|z+v|^{\frac{k}{2}}}= 2\Big[\frac{r^{\frac{k}{2}}}{|\frac{3}{2}+iy|^\frac{k}{2}} +\sum\limits_{v\geq3} \frac{r^{\frac{k}{2}}}{|z+v|^{\frac{k}{2}}}\Big],\label{eq:eq21}\end{equation}
by pulling out the term $v=2$. As $\frac{1}{|z+v|^{\frac{k}{2}}}$ is decreasing for $v\geq0$, we can bound the sum by an integral using the integral test, giving

\begin{equation}|J_1|\leq2r^{\frac{k}{2}}\Big[\frac{1}{(\frac{9}{4}+y^2)^{\frac{k}{4}}}+\int_2^\infty \frac{dt}{(t-\frac{1}{2})^2+y^2}\frac{1}{((t-\frac{1}{2})^2+y^2)^{\frac{k}{4}-1}}\Big].\label{eq:eq22}\end{equation}
%Replacing the modulus of the denominators by their definition, we get
%\begin{equation}|J_1|\leq2r^{\frac{k}{2}}\Big[\frac{1}{((\frac{3}{2})^2+y^2)^{\frac{k}{4}}}+\int_2^\infty \frac{dt}{((t-\frac{1}{2})^2+y^2)^{\frac{k}{4}}}\Big].\label{eq:eq23}\end{equation}
%\begin{equation}|J_1|\leq2r^{\frac{k}{2}}\Big[\frac{1}{(\frac{9}{4}+y^2)^{\frac{k}{4}}}+\int_2^\infty \frac{dt}{(t-\frac{1}{2})^2+y^2}\frac{1}{((t-\frac{1}{2})^2+y^2)^{\frac{k}{4}-1}}\Big].\label{eq:eq24}\end{equation}
%As $t\geq 2$,
%\begin{equation}|J_1|\leq2r^{\frac{k}{2}}\Big[\frac{1}{(\frac{9}{4}+y^2)^{\frac{k}{4}}}+\int_2^\infty \frac{dt}{(t-\frac{1}{2})^2+y^2}\frac{1}{((\frac{3}{2})^2+y^2)^{\frac{k}{4}-1}}\Big].\label{eq:eq25}\end{equation}
Utilizing the bound on $t$ and pulling out factors that do not depend on $t$ gives

\begin{equation}|J_1|\leq2r^{\frac{k}{2}}\Big[\frac{1}{(\frac{9}{4}+y^2)^{\frac{k}{4}}}+\frac{1}{(\frac{9}{4}+y^2)^{\frac{k}{4}-1}}\int_2^\infty \frac{dt}{(t-\frac{1}{2})^2+y^2}\Big].\label{eq:eq26}\end{equation}
This integral can be bounded by an absolute constant, giving

\begin{equation}|J_1|\ll 2\Big(\frac{r^2}{\frac{9}{4}+y^2}\Big)^{\frac{k}{4}}+2r^2\Big(\frac{r^2}{\frac{9}{4}+y^2}\Big)^{\frac{k}{4}-1}.\label{eq:eq27}\end{equation}
Replacing $r$ by its definition, $r=|\frac{1}{2}+iy|=\sqrt{\tfrac{1}{4}+y^2}$, we get

\begin{equation}|J_1|\ll (\tfrac{1}{4}+y^2)\Big(\frac{\tfrac{1}{4}+y^2}{\frac{9}{4}+y^2}\Big)^{\frac{k}{4}-1},\label{eq:eq28}\end{equation}
 We now simplify the right factor, which we denote as $R$, precisely

 \begin{equation}R=\Big(\frac{\tfrac{1}{4}+y^2}{\frac{9}{4}+y^2}\Big)^{\frac{k-4}{4}}= \Big(\frac{1+\frac{1}{4y^2}}{1+\frac{9}{4y^2}}\Big)^{\frac{k-4}{4}}.\label{eq:eq29}\end{equation}
Note that $R$ is increasing for all $y\geq 0$. Thus, we only need to bound it for large $y$, specifically $y=\frac{\sqrt{k}}{\sqrt{2\log{k}}}$. To begin, we convert this to an exponential function and utilize the definition of the logarithm of a quotient. Thus, 
\begin{equation}R=\exp\Big(\frac{k-4}{4}\log\Big(\frac{1+\frac{1}{4y^2}}{1+\frac{9}{4y^2}}\Big)\Big).\label{eq:eq30}\end{equation}
%=\exp\Big(\frac{k-4}{4}\Big[\log\Big(1+\frac{1}{4y^2}\Big)-\log\Big(1+\frac{9}{4y^2}\Big)\Big]\Big).\label{eq:eq30}\end{equation}
Using the Taylor expansion of $\log(1+x)$ and simplifying gives

 %\begin{equation}R=\exp\Big(\frac{k-4}{4}\Big[\frac{1}{4y^2}+O(\frac{1}{y^4})-\frac{9}{4y^2}+O(\frac{1}{y^4})\Big]\Big)\label{eq:eq31}.\end{equation}
 %We now simplify to find
 
 \begin{equation}R=\exp\Big(-\frac{k}{2y^2}+\frac{2}{y^2}+O(\frac{k}{y^4})+O(\frac{1}{y^4})\Big)\ll\exp\Big(-\frac{k}{2y^2}+O(\frac{k}{y^4})+O(\frac{1}{y^4})\Big).\label{eq:eq32}\end{equation}
%We can bound this by 
%\begin{equation}R\ll\exp\Big(-\frac{k}{2y^2}+O(\frac{k}{y^4})+O(\frac{1}{y^4})\Big)\label{eq:eq34}.\end{equation}
 Substitute this into \eqref{eq:eq28} and let $y\leq\frac{\sqrt{k}}{\sqrt{2\log{k}}}$ to get

%\begin{equation}|J_1|\ll (\tfrac{1}{4}+y^2)\exp\Big(-\frac{k}{2y^2}+O(\frac{k}{y^4})+O(\frac{1}{y^4})\Big)\label{eq:eq34}.\end{equation}
%Let $y\leq\frac{\sqrt{k}}{\sqrt{2\log{k}}}$. This implies 

 \begin{equation}|J_1|\ll\Big(\frac{1}{4}+\frac{2k}{\log{k}}\Big)\exp\Big(-\frac{\log{k}}{4}+O(\frac{k}{y^4})+O(\frac{1}{y^4})\Big).\label{eq:eq35}\end{equation}
The error terms become unnecessary for $k$ large, implying that

 \begin{equation}|J_1|\ll\Big(\frac{2k}{\log{k}}\Big)\exp\Big(-\frac{\log{k}}{4}\Big)= \Big(\frac{2k}{\log{k}}\Big)k^{-\frac{1}{4}}=o(1).\label{eq:eq36}\end{equation}
Thus $|J_1|=o(1)$ for $k$ large and $y\leq\frac{\sqrt{k}}{\sqrt{2\log{k}}}$. 

Next we bound $J_2$, which is simpler. By the triangle inequality,

 \begin{equation}|J_2|\leq\sum\limits_{(u,2v)=1, u>1} \frac{|r^{\frac{k}{2}}e^{\frac{\pi i k}{4}}(\frac{-v}{u})\epsilon_u^{k}|}{|uz+v|^{\frac{k}{2}}}\leq\sum\limits_{u\geq 3}r^{\frac{k}{2}}\sum\limits_{v\in\mathbb{Z}} \frac{1}{((v+ux)^2+u^2y^2)^{\frac{k}{2}}}.\label{eq:eq37}.\end{equation}
%We can then bound this sum by allowing extra values of $v$, giving
%\begin{equation}|J_2|\leq\sum\limits_{u\geq 3}r^{\frac{k}{2}}\sum\limits_{v\in\mathbb{Z}} \frac{1}{((v+ux)^2+u^2y^2)^{\frac{k}{2}}}.\label{eq:eq38}\end{equation}
Consider the inner sum, which we denote as $S(v)$. Let $f(v)=\frac{1}{((v+ux)^2+u^2y^2)^{\frac{k}{2}}}$ and note that $f$ is nonnegative, initially increasing, and eventually decreasing.  By Eq. (3.6) in [7],  these properties imply that $\sum\limits_{v\in\mathbb{Z}}f(v)\leq  \max{f(v)}+\int_{-\infty}^{\infty} f(t)dt$. Thus,

 \begin{equation}S(v)\leq  f(-ux)+\int_{-\infty}^{\infty}\frac{dt}{((t+ux)^2+u^2y^2)^{\frac{k}{2}}}= \frac{1}{(u^2y^2)^{\frac{k}{2}}}+\int_{-\infty}^{\infty}\frac{dt}{(t^2+u^2y^2)^{\frac{k}{2}}},\label{eq:eq39}\end{equation}
as $+ux$ was merely a shift. Let $t=uyw$ for some $w$ and note that $dt=uydw$. This gives

 \begin{equation} S(v)\leq \frac{1}{(uy)^{k}}+\frac{uy}{(uy)^k}\int_{-\infty}^{\infty}\frac{dw}{(1+w^2)^{\frac{k}{2}}}.\label{eq:eq41}\end{equation}
 We can evaluate this integral in terms of $\Gamma-$functions using a table or Mathematica. From here, Stirling's Formula implies that this integral equals  $O(k^{\frac{-1}{2}})$. For more details, refer to integral $I_k$ on page 14 of [7]. This implies that 
 
\begin{equation}S(v) \ll \frac{1}{(uy)^{k}}+\frac{1}{(uy)^{k-1}}\frac{1}{\sqrt{k}}.\label{eq:eq42}\end{equation}
Substituting this into our outer sum, we find that

\begin{equation}|J_2|\ll\sum\limits_{u\geq 3} r^{\frac{k}{2}}\Big[\frac{1}{(uy)^{k}}+\frac{1}{(uy)^{k-1}}\frac{1}{\sqrt{k}}\Big]\ll r^{\frac{k}{2}}\Big[\frac{1}{(3y)^{k}}+\frac{1}{(3y)^{k-1}}\frac{1}{\sqrt{k}}\Big].\label{eq:eq43}\end{equation}
Simplifying gives
\begin{equation}|J_2|\ll\Big(\frac{\sqrt{\tfrac{1}{4}+y^2}}{9y^2}\Big)^\frac{k}{2}+\Big(\frac{\tfrac{1}{4}+y^2}{9y^2}\Big)^\frac{k}{2} \frac{3y}{\sqrt{k}}\ll\Big(\frac{\sqrt{\tfrac{1}{4}+y^2}}{9y^2}\Big)^\frac{k}{2}\Big(\frac{3y}{\sqrt{k}}\Big).\label{eq:eq44}\end{equation}
As the first factor of this is decreasing in y, we can bound it by the value at y=1/2. 

\begin{equation}|J_2|\ll\Big(\frac{1}{324y^4}+\frac{1}{81y^2}\Big)^\frac{k}{4}\Big(\frac{3y}{\sqrt{k}}\Big)\leq\Big(\frac{8}{81}\Big)^\frac{k}{4}\Big(\frac{3y}{\sqrt{k}}\Big),\label{eq:eq45}\end{equation}
%\begin{equation}|J_2|\leq\Big(\frac{8}{81}\Big)^\frac{k}{4}\Big(\alpha+\frac{3y\beta\rho}{\sqrt{k}}\Big).\label{eq:eq46}\end{equation}
%By pulling out a factor of $\alpha$ or $3\beta\rho$, depending on which is largest, we get
%\begin{equation}|J_2|\ll\Big(\frac{8}{81}\Big)^\frac{k}{4}\Big(1+\frac{y}{\sqrt{k}}\Big).\label{eq:eq47}\end{equation}
as $y\geq\frac{1}{2}$. Assuming $y\leq o(\sqrt{k})$, this gives

\begin{equation}|J_2|\ll\Big(\frac{8}{81}\Big)^\frac{k}{4},\label{eq:eq48}\end{equation} 
as desired.
\end{proof}

Note that, as $e^{\frac{\pi ik}{4}}r^\frac{k}{2}E_0(z)=M_0+J_1+J_2$, and $e^{\frac{\pi ik}{4}}r^\frac{k}{2}E_0(z), M_0\in\mathbb{R}$, it must be the case that $(J_1+J_2)\in\mathbb{R}$. This will be used in Section \ref{sec:Sample}.

%%%%%%%%%
%%%%%%%%%
%%%%%%%%%
\section{Approximations of $E_\frac{1}{2}$} \label{sec:Eonehalf}
This section is similar to Section \ref{sec:Ezero}, except that we work with $E_\frac{1}{2}(\tfrac{1}{2}+iy)$ instead of $E_0(-\tfrac{1}{2}+iy)$.

\subsection{Properties of $E_\frac{1}{2}(\tfrac{1}{2}+iy)$}
\begin{proposition}
The function $E_{\frac{1}{2}}(\tfrac{1}{2}+iy)$ is real--valued. 
\end{proposition}
\begin{proof}
Recall the automorphy relation \eqref{eq:EzeroEonehalfRelation} from Section \ref{sec:Intro}. Letting $z=-\frac{1}{2}+iy$, this relation becomes
\begin{equation}E_{\frac{1}{2}}(-\tfrac{1}{2}+iy)=(2iy)^{-\frac{k}{2}}E_0(\tfrac{1}{2}+iy'),\label{eq:eq49}\end{equation}
for some $y'\in\mathbb{R}$. Utilizing the periodicity of $E_0$, the right side changes to
\begin{equation}L:=e^{\frac{\pi ik}{4}}(2iy)^{\frac{k}{2}}E_{\frac{1}{2}}(-\tfrac{1}{2}+iy)=e^{\frac{\pi ik}{4}}E_0(-\tfrac{1}{2}+iy').\label{eq:eq50}\end{equation}

From Proposition \ref{prop:Rzero}, the right side is real--valued and thus $L$ must also be. Utilizing the automorphy relation \eqref{eq:EonehalfEonehalfRelation} from Section \ref{sec:Intro} and simplifying gives
\begin{equation}L=i^{k}e^{\frac{\pi ik}{4}}(2ye^{\frac{\pi i}{2}})^{\frac{k}{2}}E_{\frac{1}{2}}(\tfrac{1}{2}+iy)=(-1)^k(2y)^{\frac{k}{2}}E_{\frac{1}{2}}(\tfrac{1}{2}+iy).\label{eq:eq51}\end{equation}

As noted earlier, $L$ is real--valued,  so $E_{\frac{1}{2}}(z)$ must real--valued when $x=\frac{1}{2}$. 
\end{proof}

In the next proposition, we aim to find a trigonometric approximation for $E_{\frac{1}{2}}(\tfrac{1}{2}+iy)$, as defined in Eq. \eqref{eq:EzeroAndEonehalfDef}, that is accurate for $k$ large enough. We use methods similar to that of Section \ref{sec:Ezero}. Thus, consider the following terms, which are the largest terms when $x=\frac{1}{2}$. Letting $c=0$ and $d=1$, we get the term $\frac{1}{z^{\frac{k}{2}}}$. Letting $c=-1$ and $d=1$, we get the term $\frac{1}{(z-1)^{\frac{k}{2}}}$. Similar to the case of $E_0(z)$, we want to multiply our approximation by $r^\frac{k}{2}$ where $r=|\frac{1}{2}+iy|$. Thus, let

\begin{equation} M_\frac{1}{2}(z)=r^\frac{k}{2}\Big(\frac{1}{z^{\frac{k}{2}}}+\frac{1}{(z-1)^{\frac{k}{2}}}\Big).\label{eq:MonehalfDef}\end{equation}

The following two sums include all of the terms left to be bounded. Again, we want to multiply the sums by $r^\frac{k}{2}$, giving

\begin{equation}N_1=\sum\limits_{c\neq 0,1} \frac{r^\frac{k}{2}e^{-\frac{\pi ik}{4}}G(\frac{1-2c}{8})^{k}}{(z+c)^{\frac{k}{2}}},\hspace{2cm} N_2=\sum\limits_{\substack{(d,2c)=1\\ d>1}} \frac{r^\frac{k}{2}e^{-\frac{\pi ik}{4}}G(\frac{d-2c}{8d})^{k}}{(dz+c)^{\frac{k}{2}}}.\label{eq:L1AndL2Def}\end{equation}

By these definitions,  $r^\frac{k}{2}E_\frac{1}{2}(z)=M_\frac{1}{2}(z)+N_1+N_2$.

\begin{proposition}
Let $\theta=\arctan{(2y)}$ and consider $\frac{1}{2}\leq y\leq\frac{\sqrt{k}}{\sqrt{2\log{k}}}$ . For $k$ large, $M_{\frac{1}{2}}(\tfrac{1}{2}+iy)=2\cos{(\tfrac{\theta k}{2})}$,
$|N_1|=o(1)$, and $|N_2|\ll(\frac{8}{81})^\frac{k}{4}$.
\label{lem:EonehalfApproxes}
\end{proposition}
\begin{proof}
%The proof is given in the following two subsections.
%%%
%%%
%%%
%\subsection{Simplifying $M_\frac{1}{2}(\tfrac{1}{2}+iy)$}
Using methods similar to those in Subsection \ref{Mzero}, it can easily be shown that $M_{\frac{1}{2}}(\tfrac{1}{2}+iy)=2\cos{(\tfrac{\theta k}{2})}$. 

Using techniques similar to those in Subsection \ref{sub:BoundEzero}, the bounds on $|N_1|$ and $|N_2|$ quickly become identical to the bounds of $|J_1|$ and $|J_2|$, respectively. Thus, $|N_1|=o(1)$ and $|N_2|\ll\big(\frac{8}{81}\big)^\frac{k}{4}$.

Note that $M_{\frac{1}{2}}(\tfrac{1}{2}+iy)\in\mathbb{R}$. Furthermore, as $r^{\frac{k}{2}}E_\frac{1}{2}(z)=M_\frac{1}{2}+N_1+N_2$, and $r^{\frac{k}{2}}E_\frac{1}{2}(z), M_\frac{1}{2}\in\mathbb{R}$, it must be the case that $(N_1+N_2)\in\mathbb{R}$. These properties will be used in Section \ref{sec:Sample}.
\end{proof}

%Next, we wish to show that $|N_1|=o(1)$ and $|N_2|\ll\big(\frac{8}{81}\big)^\frac{k}{4}$. We begin by bounding $N_1$. By the triangle inequality,
%\begin{equation}|N_1|\leq\sum\limits_{c\neq 0,-1} \frac{|r^{\frac{k}{2}}e^{-\frac{\pi ik}{4}}G(\frac{1-2c}{8})^{k}|}{|z+c|^{\frac{k}{2}}}=\sum\limits_{c\neq 0,-1} \frac{r^{\frac{k}{2}}}{|z+c|^{\frac{k}{2}}}.\label{eq:eq63}\end{equation}

%This is an equivalent sum to that of Equation \eqref{eq:eq19}, which has been bounded in Subsection \ref{sub:BoundEzero}, so $|N_1|=o(1)$ for large $k$ and $y\leq\frac{\sqrt{k}}{\sqrt{2\log{k}}}$. Next we bound $N_2$.

%\begin{equation}|N_2|\leq\sum\limits_{\substack{(d,2c)=1\\ d>1}} \frac{|r^{\frac{k}{2}}e^{-\frac{\pi ik}{4}}G(\frac{d-2c}{8d})^{k}|}{|dz+c|^{\frac{k}{2}}}=\sum\limits_{\substack{(d,2c)=1\\ d>1}} \frac{r^{\frac{k}{2}}}{|dz+c|^{\frac{k}{2}}}.\label{eq:eq64}\end{equation}
%From here, the work is the same as in bounding $J_2$ in Subsection \ref{sub:BoundEzero}, so $|N_2|\ll\big(\frac{8}{81}\big)^\frac{k}{4}$.\end{proof}

%Note that, as $r^{\frac{k}{2}}E_\frac{1}{2}(z)=M_\frac{1}{2}+N_1+N_2$, and $r^{\frac{k}{2}}E_\frac{1}{2}(z), M_\frac{1}{2}\in\mathbb{R}$, it must be the case that $(N_1+N_2)\in\mathbb{R}$. This will be used in Section \ref{sec:Sample}.

%%%%%%%%%
%%%%%%%%%
%%%%%%%%%
\section{Sample Points}\label{sec:Sample}
\subsection{Sample Points for $M_0(-\tfrac{1}{2}+iy)$ and $M_\frac{1}{2}(\tfrac{1}{2}+iy)$}
In this subsection, we find sample points of $M_0(-\tfrac{1}{2}+iy)$ and $M_\frac{1}{2}(\tfrac{1}{2}+iy)$ with alternating signs. This will allow us to use the intermediate value theorem, which indicates the number of zeroes we have found for each function. These zeroes are then shown to represent zeroes of $E_0(z)$ and $E_\frac{1}{2}(z)$ in the next subsection.

Recall that $ M_0(-\tfrac{1}{2}+iy)$ is a real--valued function (where $\delta=\arctan{2y}$). Note that, as $M_0(-\tfrac{1}{2}+iy)$ approximates $e^{\frac{\pi ik}{4}}E_0(-\tfrac{1}{2}+iy)$ for $\frac{1}{2}\leq y\leq\frac{\sqrt{k}}{\sqrt{2\log{k}}}$, we can restrict $\delta$ to the interval $\frac{\pi}{4}\leq\delta\leq\arctan{\big(\frac{\sqrt{2k}}{\sqrt{\log{k}}}\big)}.$ From here on, we use the notation $y_{\text{max}}=\frac{\sqrt{2k}}{\sqrt{\log{k}}}$.%Reword

We wish to find sample points of this function that have the greatest absolute value. Thus, for $k\equiv 1\mod 4$, we want $\frac{\delta k}{2}-\frac{\pi}{4}=n\pi$ for some $n\in\mathbb{N}$. Solving for $\delta$, we find $\delta=\frac{2\pi n}{k}+\frac{\pi}{2k}$. Substituting this into our interval for $\delta$ above, we get $\frac{\pi}{4}\leq\frac{2\pi n}{k}+\frac{\pi}{2k}\leq\arctan{(y_{\text{max}})}.$ Next we solve for $n$, getting 

\begin{equation}\frac{k}{8}-\frac{1}{4}\leq n\leq \frac{k}{2\pi} \arctan{(y_{\text{max}})}-\frac{1}{4}.\label{eq:eq65}\end{equation}
Let us simplify the upper bound. By the identity $\arctan{x}=\frac{\pi}{2}-\arctan{(\frac{1}{x})}$ and the Taylor expansion of $\arctan(\frac{1}{x})$ , we have 

%\begin{equation}n\leq \frac{k}{2\pi} \Big[\frac{\pi}{2}-\arctan{(\frac{1}{y_{\text{max}}}})\Big]-\frac{1}{4}.\label{eq:eq66}\end{equation}
%Next we use the Taylor expansion of $\arctan(x)$ to get 
%\begin{equation}n\leq \frac{k}{2\pi} \Big[\frac{\pi}{2}-\Big(\frac{1}{y_{\text{max}}}+O(\frac{1}{(y_{\text{max}})^3})\Big)\Big]-\frac{1}{4}\leq \frac{k}{2\pi} \Big[\frac{\pi}{2}-O(\frac{1}{y_{\text{max}}})\Big]-\frac{1}{4},\label{eq:eq67}\end{equation}
%that is,
\begin{equation}n\leq\frac{k-1}{4}-O(\frac{k}{y_{\text{max}}}).\label{eq:eq68}\end{equation}

 As the sign of $\cos{(\frac{\delta k}{2}}-\frac{\pi}{4})$ changes every time $n$ increases, this describes at least $\frac{k}{8}-O(\frac{k}{y_{\text{max}}})$ sign changes. By the intermediate value theorem, there must be a zero of $M_0(-\tfrac{1}{2}+iy)$ between each of these sign changes, so the number of zeroes that we have found is greater than or equal to $\frac{k}{8}-O(\frac{k}{y_{\text{max}}})$. We can use the same process for $k\equiv 3 \mod 4$, finding that  $\frac{k}{8}+\tfrac{1}{4}\leq n\leq \frac{k+1}{4}+O(\frac{k}{y_{\text{max}}})$. This implies  the number of zeroes that we have found is again greater than or equal to $\frac{k}{8}-O(\frac{k}{y_{\text{max}}})$. 

Next, recall that $M_{\frac{1}{2}}(\tfrac{1}{2}+iy)=2\cos{(\frac{\theta k}{2})}$ is a real--valued function that is valid for $\frac{1}{2}\leq y\leq y_{\text{max}}$. Again, by the same process we find that $\frac{k}{8}\leq n\leq\frac{k}{4}-O(\frac{k}{y_{\text{max}}})$. Again, this implies  the number of zeroes that we have found is greater than or equal to $\frac{k}{8}-O(\frac{k}{y_{\text{max}}})$. 

\subsection{Main Theorem}
We now prove Theorem \ref{EzeroThm}.

\begin{proof}
First, we prove that each zero of $M_{0}(-\tfrac{1}{2}+iy)$ corresponds to a zero of $E_{0}(-\tfrac{1}{2}+iy)$. Recall Proposition \ref{lem:EzeroApproxes}. By our bounds, $|J_1+J_2|<2=|M_0(-\tfrac{1}{2}+iy)|$ at the sample points, so each sign change of $M_0(-\tfrac{1}{2}+iy)$ found above also corresponds to a sign change of $e^{\frac{\pi i k}{4}}r^\frac{k}{2}E_0(-\tfrac{1}{2}+iy)$. Therefore, by the intermediate value theorem,we have found greater than or equal to $\frac{k}{8}-O(\frac{k}{y_{\text{max}}})$ zeroes of $E_0(-\tfrac{1}{2}+iy)$ for $k$ large. Similarly, by Proposition \ref{lem:EonehalfApproxes}, we have found greater than or equal to $\frac{k}{8}-O(\frac{k}{y_{\text{max}}})$ zeroes of $E_\frac{1}{2}(\tfrac{1}{2}+iy)$ when $k$ is large.\end{proof}

\section*{Acknowledgements}
This work was completed in summer 2016 during an REU conducted at Texas A\&M University.  The author thanks the NSF and the Department of Mathematics at Texas A\&M for supporting the REU as well as Dr. Matthew Young for his guidance.

\end{document}